\documentclass[11pt,a4paper]{amsart}

\usepackage[margin=1in]{geometry}
\usepackage{amsmath,amssymb}
\usepackage{microtype}
\usepackage{xcolor, hyperref}

\newtheorem{theorem}{Theorem}
\newtheorem{lemma}[theorem]{Lemma}
\newtheorem{conj}[theorem]{Conjecture}

\newcommand{\Sym}{\mathcal S}
\newcommand{\e}{\mathrm e}

\def\tref#1{Theorem~$\ref{#1}$}
\def\cjref#1{Conjecture~$\ref{#1}$}

\renewcommand{\le}{\leqslant}

\title{Avoiding a partial Latin square with an isotope of a given Latin square}
\author{Stijn Cambie \and Ian M. Wanless}

\thanks{Department of Computer Science, KU Leuven Campus Kulak-Kortrijk, 8500 Kortrijk, Belgium. Research supported by FWO grants 1225224N and 1215727N. E-mail: \protect\href{mailto:stijn.cambie@hotmail.com}{\protect\nolinkurl{stijn.cambie@hotmail.com}}}

\thanks{School of Mathematics, Monash University, Vic 3800, Australia. Research supported by Australian Research Council Discovery Project DP250101611. E-mail: \protect\href{mailto:ian.wanless@monash.edu}{\protect\nolinkurl{ian.wanless@monash.edu}}}

\date{}

\begin{document}

\maketitle

\begin{abstract}
For all sufficiently large $n$, we prove that every Latin square of order
$n$ can be made to avoid any prescribed partial Latin square of order $n$
by permuting only its rows and columns.  The proof combines Hall's theorem
with elementary estimates for a random permutation.
\end{abstract}

\section{Introduction}

A \emph{partial Latin square} of order $n$ is an $n\times n$ array, whose cells are either empty or contain one symbol from a set of size $n$, such that no symbol occurs more than once in any row or column. If no cell is empty the partial Latin square is a \emph{Latin square}.  We will use $[n]=\{1,\dots,n\}$ for our set of symbols, so that it matches the indices of the rows and columns. Two partial Latin squares \emph{avoid}
one another if they have different symbols in every cell that is filled in both partial Latin squares. A Latin square $L=[L_{ij}]$ is a \emph{completion} of a partial Latin square $P=[P_{ij}]$ if $L_{ij}=P_{ij}$ whenever $P_{ij}$ is defined.

Let $\Sym_n$ denote the symmetric group acting on $[n]$. There is a natural action of $\Sym_n^3$ on Latin squares, where three 
permutations permute the rows, columns and symbols respectively. This action is called \emph{isotopism} and its orbits are \emph{isotopism classes}. An isotopism is \emph{principal} if it applies the identity permutation to symbols, and the result is a \emph{principal isotope} of the original Latin square.

The purpose of this note is to prove  that for any partial Latin square $P$ and Latin square $L$ of sufficiently large order, $P$ is avoided by at least one principal isotope of $L$. In practice what we show is that $P$ can be avoided by an isotope of $L$ where we take the column permutation to be trivial. To see that this is equivalent to the claimed result, simply write $P$ and $L$ as sets of (row,column,symbol) triples, then exchange the last two coordinates.

\begin{theorem}\label{thm:main}
For all sufficiently large $n$, if $P$ is a partial Latin square and $L$ is a
Latin square of order $n$, then there are permutations
$\pi,\sigma\in\Sym_n$ such that
\[\sigma\bigl(L(\pi(i),j)\bigr)\ne P(i,j)
\]
for every filled cell $(i,j)$ of $P$.
\end{theorem} 

Andr\'en, Casselgren and Markstr\"om~\cite{ACM} proved the following powerful result, a special case of which is that every partial Latin square of sufficiently large order is avoided by some Latin square.  

\begin{theorem}\label{thm:ACM}
There are positive constants $\alpha$ and $\beta$ such that the following is true for every positive integer $n$. Suppose that $P$ is a partial Latin square of order $n$ in which each row and column has at most $\alpha n$ non-empty cells. 
Suppose that $A$ is an $n\times n$ array of subsets of $[n]$, each of size at most $\beta n$, in which
each symbol occurs at most $\beta n$ times among the sets in each row and column. Then if no cell of $P$ contains a symbol that occurs in the corresponding cell of $A$, then there is a completion of $P$ that avoids $A$.
\end{theorem}

In \tref{thm:ACM}, a Latin square $L=[L_{ij}]$ avoids $A=[A_{ij}]$ if $L_{ij}\notin A_{ij}$ for all $i,j\in[n]$.

A Latin cube of order $n$ is an $n\times n\times n$ array in which every axis parallel line is a permutation of the same set of $n$ symbols. When building a Latin cube one layer at a time, each new layer needs to be a Latin square avoiding each of the previous layers. It is possible to become stuck (i.e.~not be able to add any new layer) by the time that half of the layers have been added \cite{BCMPW12}. It is an interesting open question as to 
whether it is possible to get stuck earlier, although \tref{thm:ACM} shows that there is a linear lower bound on the number of layers before this can happen. \tref{thm:main} shows that when building a Latin cube of large order layer by layer, the second layer can be chosen from any desired isotopism class. Another application of \tref{thm:main} will be found in the forthcoming paper \cite{CLMW}.

The proof of \tref{thm:main} rests on the following
idea.  Randomly permute the rows of $L$, and record in a bipartite
\emph{conflict graph} which symbols of $L$ can meet which symbols of $P$.
An avoiding relabelling of the symbols is precisely a perfect matching in
the bipartite complement.  If this matching does not exist, Hall's theorem
produces a large complete bipartite subgraph of the conflict graph.  We rule
this out by controlling degrees and triple common neighbourhoods.  The
numerical heart of the proof is the inequality
$
             (1-\e^{-1})^3<\e^{-1}.
$ 

\section{Random permutations}

In this section we prove some properties of random permutations that we need. We start with a lemma about the probability of a random permutation hitting disjoint matchings.

\begin{lemma}\label{lem:hit}
Fix $r$.  Let $M_1,\dots,M_r$ be pairwise edge-disjoint matchings in $K_{n,n}$, and let $\Pi$ be a uniformly random perfect matching.
Then, uniformly in the choice of the $M_i$,
\[
 \Pr\big(\Pi\cap M_i\ne\varnothing\text{ for every }i\in[r]\big)
   \le\bigl(1-\e^{-1}\bigr)^r+o(1).
\]
\end{lemma}

\begin{proof}
For $J\subseteq[r]$, put $F_J=\bigcup_{j\in J}M_j$.  Suppose that the graph $F_J$ has $m\le|J|n$ edges, and note that its maximum degree is at most $r$.  Let $q_k(F_J)$ be the number of $k$-edge matchings contained in $F_J$. For each fixed $k$,
\begin{equation}\label{eq:rook}
 q_k(F_J)
 = \frac{m^k}{k!}-O_{r,k}(n^{k-1}).
\end{equation}
Indeed, start with all $k$-tuples of edges. Only
$O_{r,k}(n^{k-1})$ of them contain two edges with a common
endpoint, given that $m=O_r(n)$.

A fixed $k$-matching is contained in $\Pi$ with probability $(n-k)!/n!=n^{-k}-O_k(n^{-k-1})$.
Consequently, Bonferroni's inequalities applied to \eqref{eq:rook} give
\begin{equation}\label{eq:avoid-union}
 \Pr(\Pi\cap F_J=\varnothing)=\e^{-m/n}+o(1).
\end{equation}
Indeed, truncate inclusion--exclusion after $K$
terms.  Each truncated sum differs from the corresponding truncation of
$\exp(-m/n)$ by $O_{r,K}(n^{-1})$, while the two adjacent Bonferroni
bounds differ by at most
\[
        \frac{r^{K+1}}{(K+1)!}+O_{r,K}(n^{-1}).
\]
First let $K\to\infty$ and then $n\to\infty$.
Inclusion--exclusion over the events
$\{\Pi\cap M_i=\varnothing\}$ using \eqref{eq:avoid-union} gives that $\Pr\big(\Pi\cap M_i\ne\varnothing\text{ for every }i\in[r]\big) $ equals 
\[
 \sum_{J\subseteq[r]}(-1)^{|J|}\e^{-|F_J|/n}+o(1)=\prod_{i=1}^r \left(1- e^{-|M_i|/n}\right)+o(1)
   \le (1-e^{-1})^r+o(1). 
\]
\end{proof}


We shall also use the standard bounded-differences inequality for functions $f$ of random
permutations $\pi$: if interchanging two values of $\pi$ changes $f(\pi)$ by at
most a fixed constant, then, for every fixed $\varepsilon>0$, there is a constant $c_\varepsilon$ such that
\[
 \Pr\bigl(|f(\pi)-\mathbb Ef(\pi)|>\varepsilon n\bigr)
     \le 2\exp(-c_\varepsilon n).                               \tag{3}
\]
For completeness, this follows by exposing the values of $\pi$ one at a time
and applying Azuma's inequality; two possible continuations can be coupled by
one transposition.

\section{The conflict graph}

We are now ready to prove our result.

\begin{proof}[Proof of \tref{thm:main}]
Choose a row permutation $\pi\in\Sym_n$ uniformly at random, viewing it as
the perfect matching $\{(i,\pi(i)):i\in[n]\}$.  Define a
bipartite graph $B=B(\pi)$ whose two vertex classes are copies of $[n]$,
representing respectively the symbols of $L$ on the left and those of $P$ on the right.  Join $x$ to
$s$ if some filled cell $(i,j)$ of $P$ satisfies
\[
       L(\pi(i),j)=x
       \quad\text{and}\quad
       P(i,j)=s.                                                \tag{4}
\]
A perfect matching in the bipartite complement $\overline B$ pairs each
$L$-symbol $x$ with a $P$-symbol $s$ to which it is not adjacent.  Defining
$\sigma(x)=s$ gives precisely a symbol permutation for which
$\sigma(L(\pi(i),j))\ne P(i,j)$ in every filled cell.  We will show that some
choice of $\pi$ makes $\overline B$ have a perfect matching.

Set
\[
                 a=1-\e^{-1}.
\]
Suppose that the symbol $s$ occurs in $P$ at
$(i_1,j_1),\dots,(i_d,j_d)$.  These row indices and these column indices are
all distinct.  For a symbol $x$ of $L$, let $u_k$ be the unique row for which
$L(u_k,j_k)=x$, and put
\[
             M_{s,x}=\{(i_k,u_k):1\le k\le d\}.
\]
This is a matching: the $i_k$ are distinct, and an equality $u_k=u_\ell$ would
put $x$ twice in one row of $L$, since $j_k\ne j_\ell$.  Moreover, (4) says
exactly that
\[
               xs\in E(B)\quad\Longleftrightarrow\quad
               \pi\cap M_{s,x}\ne\varnothing.                 \tag{5}
\]
Lemma~\ref{lem:hit}
with $r=1$ gives, uniformly in $s$ and $x$,
\begin{align*}
 \mathbb E\deg_B(s)
   &=\sum_x\Pr(\pi\cap M_{s,x}\ne\varnothing)
     \le (a+o(1))n,\\
 \mathbb E\deg_B(x)
   &=\sum_s\Pr(\pi\cap M_{s,x}\ne\varnothing)
     \le (a+o(1))n.
\end{align*}
Thus
\[
               \mathbb E\deg_B(v)\le (a+o(1))n                \tag{6}
\]
for every vertex $v$ in $B$.

Next fix three distinct symbols $x_1,x_2,x_3$ of $L$.  For each $s$, the
matchings $M_{s,x_1}$, $M_{s,x_2}$ and $M_{s,x_3}$ are pairwise
edge-disjoint.
Indeed, a common edge would, by its first coordinate, arise from the same
cell of $P$, and its second coordinate would then give
$x_i=x_j$.  Lemma~\ref{lem:hit}, now with $r=3$, yields
\[
 \mathbb E\bigl|N_B(x_1)\cap N_B(x_2)\cap N_B(x_3)\bigr|
   \le (a^3+o(1))n.                                            \tag{7}
\]

Here and in (6), the $o(1)$ is uniform, as required when summing over the
symbols.

Interchanging two values of $\pi$ affects only two rows.  For a fixed
$P$-symbol, each affected row contains it at most once; for a fixed
$L$-symbol, each old or new row of $L$ contains it in exactly one column.
Thus the neighbourhood of any fixed vertex of $B$ changes in at most four
vertices.  A fixed degree is therefore $4$-Lipschitz under transpositions,
and a fixed triple common-neighbourhood size is $12$-Lipschitz.
For $n$ large, (3) followed by a union bound over the $2n$ degrees and fewer
than $n^3$ triples, shows that for every fixed $\varepsilon>0$ there is a choice of
$\pi$ for which 
\begin{align}
 \Delta(B)&\le(a+\varepsilon)n\quad\text{and},                              \tag{8}\\
 |N_B(x_1)\cap N_B(x_2)\cap N_B(x_3)|
     &\le(a^3+\varepsilon)n                                    \tag{9}
\end{align}
simultaneously for all distinct $x_1,x_2,x_3$.

Choose $\varepsilon>0$ small enough that
\[
                  a^3+\varepsilon<1-a-\varepsilon;             \tag{10}
\]
this is possible because
$ 
       (1-\e^{-1})^3<\e^{-1}$.
Fix a row permutation satisfying (8) and (9).  If $\overline B$ had no
perfect matching, Hall's theorem would give a set $X$ in its left class with
$|N_{\overline B}(X)|<|X|$.  Taking $Y=\cap_{x\in X}N_B(x)$ to be  
the complement of $N_{\overline B}(X)$  
in the right class, we obtain
\[
    X\times Y\subseteq E(B)
    \quad\text{and}\quad |X|+|Y|>n.              \tag{11}
\]
By (8), both $|X|$ and $|Y|$ are at most $(a+\varepsilon)n$; hence (11) gives
\[
    \min(|X|,|Y|)>(1-a-\varepsilon)n.                  \tag{12}
\]
For large $n$, we can choose three distinct vertices $x_1,x_2,x_3\in X$. Then (12), the definition of $Y$ and
(9) together imply
\[
 (1-a-\varepsilon)n
   <|Y|
   \le |N_B(x_1)\cap N_B(x_2)\cap N_B(x_3)|
   \le(a^3+\varepsilon)n,
\]
a contradiction of (10).  Thus $\overline B$ has a perfect matching.  As observed
above, its matching map is the required symbol permutation $\sigma$.
\end{proof}

\section{Concluding remarks}

The requirement for $n$ to be sufficiently large cannot be omitted from \tref{thm:main} as shown by the following simple example. The partial Latin square
\[\left[
\begin{array}{ccc}
1&2&\cdot\\
2&1&3  \\
\cdot&3&1\\ 
\end{array}
\right]
\]
is not avoided by any Latin square of order $3$. However, we conjecture different behaviour in the case when $P$ is itself a Latin square.

\begin{conj}\label{cj:alln}
    If $L$ and $L'$ are any Latin squares of the same order $n>1$, then there is a principal isotope of $L'$ that avoids $L$.
\end{conj}

Note that~\cjref{cj:alln} is trivially true when $L'=L$, because any cyclic permutation on the rows will work. Hence it is also true when $L'$ is a principal isotope of $L$. In particular, \cjref{cj:alln} is true when $n\le3$, which shows a difference between Latin squares and partial Latin squares. By computation, we have confirmed that \cjref{cj:alln} holds for $n\le7$. To prove it holds for a particular value of $n$, it suffices to test all possible pairs $(L,L')$ where $L$ occurs before $L'$ in a list of principal isotopism class representatives of order $n$. The number of principal isotopism classes is A123234 on the OEIS \cite{OEIS}. There are $2\,420\,400$ principal isotopism classes of order $7$.

In the spirit of \tref{thm:ACM}, we also offer this:

\begin{conj}\label{cj:genACM}
For any constant $t$ 
the following is true for every sufficiently large integers $n$. Suppose that $L$ is a Latin square of order $n$. 
Suppose that $A$ is an $n\times n$ array of subsets of $[n]$, each of size at most $t$, in which
each symbol occurs at most $t$ times among the sets in each row and column. Then there is a principal isotope of $L$ that avoids $A$.
\end{conj}

If true, \cjref{cj:genACM} would mean that when stacking layers on top of each other to try to build a Latin cube, we would be free to choose the isotopism class of each layer up to $t+1$ layers. We have been somewhat cautious in our framing of \cjref{cj:genACM} because some generalisations of \tref{thm:ACM} fail quite badly. Let $L$ be the Cayley table of the elementary abelian group of order
$n=2^k$ for arbitrarily large $k$. Let $P=[P_{ij}]$ be a partial Latin square of order $n$ with just three filled cells, namely $P_{11}=P_{22}=1$ and $P_{12}=2$. Suppose that $A=[A_{ij}]$ is an $n\times n$ array of sets, all of which are empty except that $A_{21}=\{2\}$. Then there is no completion of $P$ that avoids $A$ and is in the isotopism class of $L$.

\section*{AI declaration and acknowledgement}

This problem was proposed during the MATRIX-MFO Tandem Workshop (Combinatorial Interchange). The main ideas in the proof of \tref{thm:main} were found by ChatGPT 5.6 and then developed by the authors.

\end{document}